\documentclass[a4paper,10pt]{article}

\usepackage{graphicx}
\usepackage{geometry}
\usepackage{amsthm}
\usepackage{amssymb}
\usepackage{amsmath}
\usepackage{hyperref}
\usepackage{xcolor}
\usepackage{enumerate}
\usepackage{listings}
\usepackage{complexity}
\usepackage{blkarray}
\usepackage{lmodern}

\usepackage[font=small,labelfont=bf]{caption}
\usepackage[font=small,labelfont=normalfont,labelformat=simple]{subcaption}

\graphicspath{{figs/}}

\newtheorem{theorem}{Theorem}

\newtheorem{lemma}[theorem]{Lemma}
\newtheorem{corollary}[theorem]{Corollary}

\newtheorem{observation}{Observation}
\newtheorem{proposition}{Proposition}
\newtheorem{conjecture}{Conjecture}
\newtheorem{problem}{Problem}

\author
{
Michael Anastos \thanks{Institute of Mathematics, Freie Universit\"{a}t Berlin, Germany, email: \texttt{manastos@zedat.fu-berlin.de}.}
\and
Ander Lamaison \thanks{Institute of Mathematics, Freie Universit\"{a}t Berlin, Germany, email: \texttt{lamaison@math.fu-berlin.de}. Funded by the Deutsche Forschungsgemeinschaft (DFG, German Research Foundation) under Germany's Excellence Strategy - The Berlin Mathematics Research Center MATH+ (EXC-2046/1, project ID: 390685689).}
\and
Raphael Steiner \thanks{Institute of Mathematics, Technische Universit\"at Berlin, Germany, email: \texttt{steiner@math.tu-berlin.de}.
Funded by DFG-GRK 2434 Facets of Complexity.}
\and
Tibor Szab\'{o} \thanks{Institute of Mathematics, Freie Universit\"{a}t Berlin, Germany, email: \texttt{szabo@math.fu-berlin.de}. 
Research supported in part by GIF grant G-1347-304.6/2016.}
}

\date{\today}

\title{Majority Colorings of Sparse Digraphs}

\begin{document}
\maketitle

\begin{abstract}
A \emph{majority coloring} of a directed graph is a vertex-coloring in which every vertex has the same color as at most half of its out-neighbors. Kreutzer, Oum, Seymour, van der Zypen and Wood \cite{kreutzer} proved that every digraph has a majority $4$-coloring and conjectured that every digraph admits a majority $3$-coloring. 
We verify this conjecture for digraphs with chromatic number at most $6$ or dichromatic number at most $3$. We obtain analogous results for list coloring: We show that every digraph with list chromatic number at most $6$ or list dichromatic number at most $3$ is majority $3$-choosable. We deduce that digraphs with maximum out-degree at most $4$ or maximum degree at most $7$ are majority $3$-choosable.
On the way to these results we investigate digraphs admitting a majority $2$-coloring. We show that every digraph without odd directed cycles is majority $2$-choosable. We answer an open question posed in \cite{kreutzer} negatively, by showing that deciding whether a given digraph is majority $2$-colorable is NP-complete.  
Finally we deal with a fractional relaxation of majority coloring proposed in \cite{kreutzer} and show that every digraph has a fractional majority 3.9602-coloring. We show that every digraph $D$ with minimum out-degree $\Omega\left((1/\varepsilon)^2\ln(1/\varepsilon)\right)$ has a fractional  majority $(2+\varepsilon)$-coloring. 
\end{abstract}

\section{Introduction}
\paragraph*{Preliminiaries.}
Digraphs in this paper are considered loopless, have no parallel edges, but are allowed to have anti-parallel pairs of edges (\emph{digons}). A directed edge with tail $u$ and head $v$ is denoted by $(u,v)$. For a digraph $D$ and a vertex $v \in V(D)$, we let $N^+(v),$ $N^-(v)$ denote the out- and in-neighborhood of $v$ in $D$ and $d^+(v)$, $d^-(v)$ the respective sizes. We denote by $\delta^+(D)$, $\delta^-(D)$, $\Delta^+(D)$, $\Delta^-(D)$ the minimum or maximum  out- or in-degree of $D$, respectively, and let $\Delta(D)=\max\{d^+(v)+d^-(v)|v \in V(D)\}$ denote the maximum degree in $D$. The \emph{underlying graph} of a digraph $D$, denoted by $U(D)$, is the simple undirected graph with vertex set $V(D)$ in which two vertices $x \neq y$ are adjacent iff $(x,y) \in E(D)$ or $(y,x) \in E(D)$. We say that $D$ is $r$-regular for an integer $r \ge 1$ if $d^+(x)=d^-(x)=r$ for every $x \in V(D)$. 

A \emph{majority coloring} of a digraph $D$ with $k$ colors is an assignment $c:V(D) \rightarrow \{1,\ldots,k\}$ such that for every $v \in V(D)$, we have $c(w)=c(v)$ for at most half of all out-neighbors $w \in N^+(v)$. This notion of coloring was first introduced and studied by Kreutzer, Oum, Seymour, van der Zypen, and Wood~\cite{kreutzer}, who showed that every digraph has a majority $4$-coloring. Their elegant argument is based on the observation that every acyclic digraph can be majority $2$-colored. The relevant property of acyclic digraphs is that there is an ordering of its vertices, in which every vertex is preceded by its complete out-neighborhood. Then coloring vertices along this ordering with two colors such that for each vertex we use the color that appears least frequently in the (already colored) out-neighborhood will produce a majority $2$-coloring.

It is easy to construct digraphs which do require $3$ colors for a majority coloring. The canonical examples are the odd directed cycles $\vec{C}_{2k+1}, k \ge 1$, which are not majority $2$-colorable since for digraphs with maximum out-degree one majority-coloring and proper graph coloring of the underlying graph are equivalent. 
However, no example of a digraph is known that requires the use of four colors. 
Kreutzer et al. conjectured that there are none.
\begin{conjecture}[\cite{kreutzer}]\label{mainconj}
Every digraph is majority $3$-colorable.
\end{conjecture}
Kreutzer et al.~\cite{kreutzer} also provide ample evidence for their conjecture by establishing that it holds 
for ``most'' digraphs. 
They show, using the Lov\'{a}sz Local Lemma, that the uniform random $3$-coloring is a majority $3$-coloring with non-zero probability if certain local density conditions hold, namely if
\begin{itemize}
\item $\delta^+(D) >72 \ln (3|V(D)|)$, or
\item $\delta^+(D) \ge 1200$ and $\Delta^-(D) \leq \frac{\exp(\delta^+(D)/72)}{12\delta^+(D)}$,
\end{itemize} 
In \cite{kreutzer} it is also mentioned at the end that a more careful analysis of the Local Lemma approach works for $r$-regular digraphs provided $r \ge 144$. 
Subsequently Gir\~{a}o, Kittipassorn, and Popielarz~\cite{girao} studied tournaments in particular, and showed, also using the probabilistic method, that if the minimum out-degree is at least $55$, then the tournament is 
majority $3$-colorable. 

These are all the results  we are aware of about Conjecture~\ref{mainconj}. All the proofs use some probabilistic idea and require some lower bound on the minimum out-degree. However,
digraphs with small minimum out-degree seem to be outside the realm of any such probabilistic methods and it looks like they constitute a main difficulty of the problem. 
This is also illustrated by the fact that it was not even known whether planar digraphs are majority $3$-colorable. 

In this paper our main motivation is to complement the existing results for locally dense digraphs and provide approaches to this intriguing conjecture on the opposite end of the spectrum.


\subsection{Our results}

\subsubsection{Majority $3$-Colorability}
Since a proper coloring is also a majority coloring, Conjecture~\ref{mainconj} is immediately true for digraphs with chromatic number at most three. For four-chromatic digraphs this is already not obvious. 
Our first result resolves the conjecture for digraphs with low chromatic number, including planar digraphs.
\begin{theorem} \label{chrom6}
Let $D$ be a digraph such that $\chi(D) \leq 6$. Then $D$ is majority $3$-colorable.
\end{theorem}

Another coloring concept for digraphs which greatly grew in importance in the last two decades is the \emph{dichromatic number}. This parameter of a digraph $D$, denoted by $\vec{\chi}(D)$, was introduced already in 1982 by Victor Neumann-Lara~\cite{neulara}.\footnote{Some authors simply refer to $\vec{\chi}(D)$ as the chromatic number of the
digraph $D$.} 
It is defined as the smallest $k$ allowing a partition $X_1,\ldots,X_k$ of $V(D)$ such that $D[X_i]$ is acyclic for all $i=1,\ldots,k$. 

In the introduction above we mentioned how to give a majority $2$-coloring of acyclic digraphs, i.e. digraphs 
with dichromatic number one. 
In our second main result we prove Conjecture~\ref{mainconj} for digraphs with dichromatic number at most three.

\begin{theorem}\label{dichrom3}
Let $D$ be a digraph such that $\vec{\chi}(D)\leq 3$. Then $D$ is majority $3$-colorable. 
\end{theorem}

The results of \cite{kreutzer} and \cite{girao} cited in the introduction indicate that the case of $r$-regular digraphs for constant $r$ is probably the most important benchmark in the study of Conjecture~\ref{mainconj}. Recall in particular that the Local Lemma approach works for $r$-regular digraphs provided $r \ge 144$. 
Note however the crucial non-monotonicity in the problem: even though we do not know whether Conjecture~\ref{mainconj} is true for $r=143$, it does hold (quite easily) for $r=1$ and $2$. 
Indeed, a $1$-regular digraph is the disjoint union of directed cycles, and hence we can $3$-color it properly to obtain a majority-coloring. Then Conjecture~\ref{mainconj} also follows for $2$-regular digraphs. Even more generally, the validity of the conjecture for any odd regularity $r-1$ implies it for the next even regularity $r$.
This is the consequence of the fact that for even $r$ any $r$-regular digraph $D$ contains a $1$-regular spanning subgraph $F$ and any $3$-majority coloring of the $(r-1)$-regular digraph $D-F$ is also a majority coloring of $D$.
Most generally, if a digraph $D' $ is obtained from a digraph $D$ by adding an edge $(u,v)$ whose tail has odd out-degree $d^+_D(u)$ then a majority coloring of $D$ is also a majority coloring of $D'$. 

From our next main result it will follow that majority $3$-colorings also exist in the cases when $r=3$ or $4$.
Since our proof of Conjecture~\ref{mainconj} for $3$-regular digraphs relies crucially on a natural list coloring extension of majority coloring and also implies a stronger statement, we first introduce this stronger concept in the next subsection and then present our results there. 

\subsubsection{Majority $3$-Choosability}
The notion of majority choosability of digraphs was already proposed in \cite{kreutzer}. We call a digraph \emph{$k$-majority-choosable}, if for any assignment of lists of size at least $k$ to the vertices, we can choose colors from the respective lists such that the arising coloring is a majority coloring. 

Anholcer, Bosek, and Grytczuk~\cite{lists} gave a beautiful proof to show that every digraph is majority $4$-choosable (not only majority $4$-colorable).

It was already noted in~\cite{kreutzer} that all their results about dense digraphs using probabilistic methods, including in particular the one about $r$-regular digraphs for $r\geq 144$, remain valid for majority $3$-choosability instead of majority $3$-colorability.
Here we obtain results at the other end of the spectrum, involving digraphs with bounded maximum (out-)degrees. 
This implies Conjecture~\ref{mainconj} for these cases, and in particular also for $3$-regular digraphs.

\begin{theorem}\label{maxdegree}
If $\Delta^+(D) \leq 4$ or $\Delta(U(D)) \leq 6$ or $\Delta(D) \leq 7$, then $D$ is majority $3$-choosable.
\end{theorem}

Next we derive choosability analogues of our first two theorems. 
The analogue of Theorem~\ref{chrom6} connects the choosability of the underlying graph to majority choosability. 
\begin{theorem} \label{lchrom6}
Let $D$ be a digraph whose underlying undirected graph is $6$-choosable. Then $D$ is majority $3$-choosable. In particular any digraph with a $5$-degenerate underlying graph is majority $3$-choosable.
\end{theorem}

The list dichromatic number $\vec{\chi}_\ell(D)$ of a digraph $D$ was introduced by Bensmail, Harutyunyan, and Le~\cite{listdichrom}. It is defined as the minimum integer $k \ge 1$ such that for any assignment of lists of size at least $k$ to the vertices, we can choose colors without producing monochromatic directed cycles. We have the following analogue of Theorem~\ref{dichrom3} involving this parameter.
\begin{theorem} \label{ldichrom3}
Let $D$ be a digraph with $\vec{\chi}_\ell(D) \leq 3$. Then $D$ is majority $3$-choosable. 
\end{theorem}

\subsubsection{Fractional Majority Colorings}
The concept of fractional majority coloring emerges as an LP-relaxation of the problem of majority coloring, much in the same way as the usual fractional colorings of graphs. This notion was already introduced in \cite{kreutzer}. 
The definition is somewhat technical and we postpone it to Section~\ref{sec:fractional}. To appreciate our results here, it is sufficient to keep in mind that the minimum total weight of a fractional majority coloring is at most the majority chromatic number. 

Kreutzer et al~ \cite{kreutzer} ask what is the smallest constant $K$ such that every digraph admits a fractional majority coloring with total weight at most $K$. This is yet another direction to approach Conjecture~\ref{mainconj} from. Proving that there is a fractional majority coloring with total weight $3$ for every digraph would certainly be an easier task. 
Here we take the first step in this direction and show that the upper bound of 4, which follows from the fact that every digraph is majority $4$-colorable, can be slightly improved.

\begin{theorem}\label{4-c}
Every digraph $D$ admits a fractional majority coloring with total weight at most $3.9602$.
\end{theorem}
Our proof is the combination of an intricate probabilistic coloring with some deterministic alteration.

In the second theorem of the section we show that digraphs with sufficiently large minimum out-degree have fractional majority colorings with total weight arbitrarily close to $2$. The results in \cite{kreutzer} obtained using the Local Lemma instead only give an upper bound of $3$ under stronger assumptions (an upper bound on the maximum degree). This result further highlights that the main difficulty of deciding Conjecture \ref{mainconj} might
lie with digraphs of low out-degrees.

\begin{theorem}\label{fracldeg}
There exists a constant $C>0$ such that for every $\varepsilon >0$ and every digraph $D$ with $\delta^+(D)\geq C(1/\varepsilon)^2\ln(2/\varepsilon)$, there exists a fractional majority coloring of $D$ with total weight at most $2+\varepsilon$.
\end{theorem}

\subsubsection{Majority $2$-Colorability}

We prepare the investigation of majority $3$-colorable digraphs in Section~\ref{sec:sparsedigraphs} with an analysis of majority $2$-colorings in Section~\ref{sec:nooddcycles}. An important special case will be digraphs which do not contain a directed cycle of odd length.
\begin{theorem}\label{corfor2choosability}
If $D$ is a digraph without odd directed cycles, then $D$ is majority $2$-choosable. 
\end{theorem}

In Section~\ref{sec:npcompleteness} we are concerned with another open question posed in \cite{kreutzer}, where it was asked whether there is a characterisation of digraphs that have a majority 2-coloring (or a polynomial time algorithm to recognise such digraphs). Our last theorem answers this (most likely) in the negative.

\begin{theorem}\label{thm:nphardness}
Deciding whether a given digraph $D$ is majority $2$-colorable is NP-complete.
\end{theorem}

\bigskip

\bigskip 

\paragraph{Organization of the paper.} 
In Section~\ref{sec:nooddcycles} we obtain Theorem~\ref{corfor2choosability} as a consequence of a more general result (Theorem~\ref{theorem:2choosability}). This result is crucial for the proofs of Theorems~\ref{chrom6},~\ref{dichrom3},~\ref{maxdegree},~\ref{lchrom6},~\ref{ldichrom3}, which are presented in Section~\ref{sec:sparsedigraphs}. In Section~\ref{sec:fractional} we treat fractional majority colorings and prove Theorems~\ref{4-c} and~\ref{fracldeg}. Finally, we present the proof of Theorem~\ref{thm:nphardness} in Section~\ref{sec:npcompleteness} and conclude with some open problems in Section~\ref{concl}.

\section{Digraphs without Odd Directed Cycles} \label{sec:nooddcycles}
We have seen that acyclic digraphs as well as bipartite digraphs are majority $2$-colorable. We have also seen that odd directed cycles are canonical examples of digraphs having no majority $2$-coloring. It is therefore natural to try unifying these results and ask whether every digraph without an odd directed cycle is majority $2$-colorable. In this section, we answer this question positively. We start with a simple observation:
\begin{lemma}\label{lemma:strongcompbipartite}
A digraph $D$ contains no odd directed cycles if and only if all its strong components are bipartite.
\end{lemma}
\begin{proof}
Sufficiency of this condition is obvious, as a directed cycle is always contained in a single strong component. For the reverse direction, it suffices to observe that if $D$ is strongly connected and all directed cycles have even length, then $D$ is bipartite. However, this statement can be easily verified by considering an ear decomposition of $D$.
\end{proof}
\begin{proposition}\label{proposition:nooddcycles}
Let $D$ be a digraph which contains no odd directed cycles. Then $D$ is majority $2$-colorable. Moreover, such a coloring can be chosen to extend any given pre-coloring of the sinks of $D$ with colors $1,2$.
\end{proposition}
\begin{proof}
We prove the statement by induction on the number $s \ge 1$ of strong components of $D$. Suppose first that $s=1$, i.e. $D$ is strongly connected. Then by Lemma~\ref{lemma:strongcompbipartite} $D$ is bipartite and therefore majority $2$-colorable. As $D$ is either a single vertex or contains no sinks, the claim follows.

Now let $s \ge 2$ and suppose that the statement holds true for all digraphs with at most $s-1$ strong components. We now distinguish two cases: Either, $D$ is an independent set of $s$ vertices, and therefore, the claim holds trivially true. If there exists at least one arc in $D$, there has to be a strong component of $D$ containing no sinks such that there are no arcs entering it. Let $X$ be the vertex set of this component. 

Now let a pre-coloring of the sinks of $D$ with $1,2$ be given. By the choice of $X$, $D-X$ has the same set of sinks as $D$ and $s-1$ strong components. By the inductive assumption, there exists a majority $2$-coloring $c:V(D) \setminus X \rightarrow \{1,2\}$ of $D-X$ which extends the pre-coloring of the sinks. By Lemma~\ref{lemma:strongcompbipartite}, there exists a bipartition $\{A,B\}$ of $D[X]$. 

For any subset $W \subseteq X$ equipped with a vertex-coloring $c_W:V(D) \setminus W \rightarrow \{1,2\}$ of $D-W$, any vertex $x \in W$, and any $i \in \{1,2\}$, denote by $d(c_W,i,x)$ the number of out-neighbors of $x$ which lie in $V(D) \setminus W$ and have color $i$ under $c_W$. 

We now claim that there exists a subset $U \subseteq X$ and a $1,2$-coloring $c_{U}$ of $D-U$ which extends $c$, such that 

\begin{itemize}

\item Every vertex $x \in V(D) \setminus U$ has at least $\frac{d^+(x)}{2}$ out-neighbors in $V(D) \setminus U$ with a color different from $c_U(x)$. 
\item Every vertex $x \in U$ fulfills $\max\{d(c_{U},1,x),d(c_{U},2,x)\} < \frac{1}{2}d^+(x)$.

\end{itemize}

In order to find such a set, we apply the following procedure:

We keep track of a pair $(W,c_W)$, consisting of a subset $W \subseteq X$ and a vertex-coloring $c_W:V(D) \setminus W \rightarrow \{1,2\}$ extending $c$. As an invariant we will keep the first of the two above properties, i.e. we assert that every vertex $x \in V(D) \setminus W$ has at least $\frac{d^+(x)}{2}$ out-neighbors with a different color according to $c_W$. 

We initialize $W:=X$, $c_W:=c$. It is clear that this assignment satifies the invariant (remember that $c$ is a majority coloring of $D-X$, and that there are no edges entering $X$).

As long as a vertex $x_0 \in W$ with $\max\{d(c_{U},1,x_0),d(c_{U},2,x_0)\} \ge \frac{1}{2}d^+(x_0)$ exists, we choose such a vertex. We put $W':=W \setminus \{x_0\}$, and define a coloring $c_{W'}$ of $D-W'$ according to  \[c_{W'}(x):=\begin{cases} c_W(x), & \text{if } x \neq x_0 \cr 1, & \text{if } x=x_0, d(c_W,1,x_0)<d(c_W,2,x_0) \cr 2, & \text{if } x=x_0, d(c_W,1,x_0) \ge d(c_W,2,x_0) \end{cases}.\] 

It is easily verified that the coloring $c_{W'}$ also fulfills the invariant, since by definition $x_0$ has at least $\max\{d(c_{U},1,x_0),d(c_{U},2,x_0)\} \ge \frac{1}{2}d^+(x_0)$ out-neighbors in $D-W'$ of different color. Furthermore, for every vertex $x \in W$, the number of out-neighbors of different color does not decrease by coloring $x_0$.

Finally we update according to $(W,c_W):=(W',c_{W'})$.

In the moment the procedure terminates, we have found a subset $U:=W \subseteq X$ and a $1,2$-coloring $c_U$ of $D-U$ extending $c$ with the property that every vertex $x \in V(D) \setminus U$ has at least $\frac{d^+(x)}{2}$ out-neighbors with different color according to $c_U$. Since the procedure terminated, we furthermore have $\max\{d(c_{U},1,x),d(c_{U},2,x)\} < \frac{1}{2}d^+(x)$ for every vertex $x \in U$. This shows that $U$ satisfies both of the conditions stated above.

We now finally extend the coloring $c_U$ of $V(D) \setminus U$ to a $1,2$-coloring of $D$ by giving color $1$ to each vertex in $A \cap U$ and color $2$ to every vertex in $B \cap U$. This coloring extends $c$ and therefore the initial pre-coloring of the sinks, and is a majority coloring: By the first of the two conditions, every vertex $x \in V(D) \setminus U$ has at least $\frac{d^+(x)}{2}$ out-neighbors with a different color. For each vertex $x \in U$, since $\{A,B\}$ is a bipartition of $D[X]$, all out-neighbors in $U$ have a different color, and among the out-neighbors in $D-X$, at most $\max\{d(c_{U},1,x),d(c_{U},2,x)\} < \frac{1}{2}d^+(x)$ can share its color. Therefore every vertex satifies the condition for a majority-coloring, and this concludes the proof of the claim.
\end{proof}

\begin{theorem}\label{theorem:2choosability}
Let $D$ be a digraph and for each $v \in V(D)$ let $L(v)$ be an assigned list of two colors. Suppose that there exists no odd directed cycle in $D$ all whose vertices are assigned the same list. Then there is a majority-coloring $c$ of $D$ such that $c(v) \in L(v)$ for all $v \in V(D)$.
\end{theorem}
\begin{proof}
We may assume w.l.o.g. that color lists of adjacent vertices always intersect: Otherwise, we remove all edges between vertices with disjoint color lists to obtain a digraph $D'$. Any majority-coloring of $D'$ with colors chosen from the lists will also be a majority-coloring of $D$. 

Now consider an arbitrary pair $\{a,b\}$ of colors and let $X_{\{a,b\}}:=\{x \in V(D)|L(x)=\{a,b\}\}$. By assumption $D[X_{\{a,b\}}]$ contains no odd directed cycles. Let $D'_{\{a,b\}}$ be the digraph obtained from $D[X_{\{a,b\}}]$ by adding all arcs $(x,y) \in E(D)$ with $x \in X_{\{a,b\}}$ and $y \notin X_{\{a,b\}}$ and their endpoints. Since we only add sinks to $D[X_{\{a,b\}}]$, also $D'_{\{a,b\}}$ contains no odd directed cycles. For each vertex $y \in N^+(X_{\{a,b\}}) \setminus X_{\{a,b\}}$, there is a unique color $p_{\{a,b\}}(y)$ in $L(y) \cap \{a,b\}$. Pre-color the sinks of $D'_{\{a,b\}}$ in such a way that every vertex $y \in N^+(X_{\{a,b\}}) \setminus X_{\{a,b\}}$ receives color $p_{\{a,b\}}(y)$. By Proposition~\ref{proposition:nooddcycles} we can now find a majority-coloring $c_{\{a,b\}}$ of $D'_{\{a,b\}}$ extending this pre-coloring with colors $a$ and $b$. 

Now define a coloring $c$ of all vertices in $D$ by setting $c(x):=c_{\{a,b\}}(x)$ if $L(x)=\{a,b\}$. Clearly, we have $c(x) \in L(x)$ for all $x \in V(D)$. We claim that $c$ is a majority-coloring of $D$. Indeed, for any vertex $x \in V(D)$, if $L(x)=\{a,b\}$, then we have $N^+(x)=N_{D'_{\{a,b\}}}^+(x)$, and $\{y \in N^+(x) \mid c(y)=c(x)\} \subseteq \{y \in N_{D'_{\{a,b\}}}^+(x) \mid c_{\{a,b\}}(x)=c_{\{a,b\}}(y)\}$. Hence, at most half of the out-neighbors of $x$ share its color, and the claim follows.
\end{proof}
Theorem~\ref{corfor2choosability} is now obtained from Theorem~\ref{theorem:2choosability} as a direct consequence.
\section{Majority $3$-Colorings of Sparse Digraphs} \label{sec:sparsedigraphs}
As a consequence of Theorem~\ref{theorem:2choosability}, we obtain our main result:
\begin{theorem}\label{theorem:3oddpartition}
Let $D$ be a digraph. Suppose there is a partition $\{X_1,X_2,X_3\}$ of the vertex set such that for every $i \in \{1,2,3\}$, $D[X_i]$ contains no odd directed cycles. Then $D$ is majority $3$-colorable. 
\end{theorem}
\begin{proof}
We assign lists of size two to the vertices of $D$, namely, we assign the list $\{2,3\}$ to all vertices in $X_1$, the list $\{1,3\}$ to all vertices in $X_2$, and the list $\{1,2\}$ to all vertices in $X_3$. Because $D[X_i], i=1,2,3$ contains no odd directed cycle, we can apply Theorem~\ref{theorem:2choosability} to conclude that there exists a majority-coloring of $D$ which uses only colors $1,2$ and $3$. This proves the claim.
\end{proof}
From this we now directly derive Theorem~\ref{chrom6} and~\ref{dichrom3}.
\begin{proof}[Proof of Theorem~\ref{chrom6}.]
If $\chi(D) \leq 6$, then $D$ admits a partition $Y_1,\ldots,Y_6$ into independent sets. Using the partition $\{Y_1 \cup Y_2,Y_3 \cup Y_4,Y_5 \cup Y_6\}$ of the vertex set to apply Theorem~\ref{theorem:3oddpartition} now shows that $D$ is indeed majority $3$-colorable.
\end{proof}

\begin{proof}[Proof of Theorem~\ref{dichrom3}.]
If $\vec{\chi}(D) \leq 3$, then there exists a partition $\{X_1,X_2,X_3\}$ of the vertex set such that $D[X_i]$ contains no directed cycles, for $i=1,2,3$. The claim now follows by Theorem~\ref{theorem:3oddpartition}.
\end{proof}

The fact that Theorem~\ref{theorem:2choosability} deals with an assignment of lists can be further exploited to show analogues of Theorem~\ref{theorem:3oddpartition}, Theorems~\ref{chrom6} and \ref{dichrom3} for list colorings. 

For this purpose we need the following notion: Call a digraph $D$ \emph{OD-$3$-choosable} if for any assignment of color lists $L(x)$, $x \in V(D)$ of size $3$ to the vertices, there exists a choice function $c$ (i.e. $c(x) \in L(x)$ for all $x \in V(D)$) such that no odd directed cycle in $(D,c)$ is monochromatic. 

\begin{theorem} \label{theorem:odd3listpartition}
Let $D$ be a digraph. If $D$ is OD-$3$-choosable, then $D$ is majority $3$-choosable.
\end{theorem}
\begin{proof}
Let $L(v)$ for all $v \in V(D)$ be a given color list of size three. We have to show that there is a majority-coloring $c$ of $D$ such that $c(v) \in L(v)$ for all $v \in V(D)$. For every $v \in V(D)$, we let $L^\ast(v):=\{\{C_1,C_2\}|C_1 \neq C_2 \in L(v)\}$ contain all three unordered color-pairs in $L(v)$. Since $D$ is OD-$3$-choosable, there exists a choice function $c^\ast$ on $V(D)$ such that $c^\ast(v) \in L^\ast(v)$ for each vertex $v \in V(D)$ is a subset of $L(v)$ of size two and such that there exists no odd directed cycle in $D$ which is monochromatic with respect to $c^\ast$. If we now consider $c^\ast(v), v \in V(D)$ as an assignment of lists of size two to the vertices of $D$, we can apply Theorem~\ref{theorem:2choosability} to conclude that there is a majority-coloring $c$ of $D$ such that $c(v) \in c^\ast(v) \subseteq L(v)$ for every vertex $v \in V(D)$. As $L(\cdot)$ was arbitrary, we conclude that $D$ is majority $3$-choosable.
\end{proof}

We are now ready to prove Theorem~\ref{lchrom6} and Theorem~\ref{ldichrom3}.
\begin{proof}[Proof of Theorem~\ref{lchrom6}.]
We show that $D$ is OD-$3$-choosable, the claim then follows by Theorem~\ref{theorem:odd3listpartition}. Let $L(v)$ for each vertex $v \in V(D)$  be an assigned list of three colors. For each color $C$ used in one of the lists, let $C'$ be a distinct copy of this color. We now consider the assignment $L_6(\cdot)$ of lists of size $6$ to the vertices of $D$, where for each vertex $v \in V(D)$, $L_6(v):=\{C_1,C_1',C_2,C_2',C_3,C_3'\}$ if $C_1, C_2, C_3$ denote the colors contained in $L(v)$. Because the underlying graph of $D$ is $6$-choosable, there is a proper coloring $c_6$ of $D$ such that $c_6(v) \in L_6(v)$ for all $v \in V(D)$. Now consider the coloring $c$ of $D$ obtained from $c_6$ by identifying each copy $C'$ of an original color $C$ with $C$ again. We then have $c(v) \in L(v)$ for every $v \in V(D)$. Since $c_6$ was a proper coloring of the undirected underlying graph of $D$, each color class with respect to $c$ induces a bipartite subdigraph of $D$, and hence there are no monochromatic odd directed cycles in $(D,c)$. Hence, $D$ is OD-$3$-choosable.
\end{proof}
\begin{proof}[Proof of Theorem~\ref{ldichrom3}.]
This follows directly since any digraph with $\vec{\chi}_\ell(D) \leq 3$ is clearly OD-$3$-choosable.
\end{proof}

The rest of this section is devoted to proving Theorem~\ref{maxdegree}. We prepare it with the following Lemma, whose proof makes use of Theorems~\ref{lchrom6} and~\ref{ldichrom3}.
\begin{lemma}\label{outdegree3}
Let $D$ be a digraph such that $\min\{d^+(x),d^-(x)+1\} \leq 3$ for every $x \in V(D)$. Then $D$ is OD-$3$-choosable.
\end{lemma}
\begin{proof}
Suppose the claim was false and consider a counterexample $D$ minimizing $|V(D)|+|E(D)|$. We have $|V(D)| \ge 4$, $D$ is connected and every proper subdigraph of $D$ must be OD-$3$-choosable.

We first consider the case that there is a vertex $v$ with $d^-(v) \leq 2$. Since $D-v$ is OD-$3$-choosable, given any assignment $L(v), v \in V(D)$ of lists of size at least $3$ to the vertices, we can choose colors $c(w) \in L(w)$ from the lists for every $w \in V(D) \setminus \{v\}$ such that in $D-v$, there exists no monochromatic odd directed cycle. Now assign to $v$ a color $c(v) \in L(v) \setminus \{c(w) \mid w \in N^-(v)\}$. We claim that $c$ is a coloring of $D$ without monochromatic odd directed cycles. In fact, such a cycle would have to pass $v$, however no edge entering $v$ is monochromatic. Therefore $D$ is OD-$3$-choosable, a contradiction.

Hence we know for every $x \in V(D)$ that $d^-(x) \ge 3$. Since $\min\{d^+(x),d^-(x)+1\} \leq 3$, we also must have $d^+(x) \le 3$. We conclude \[3|V(D)| \le \sum_{v \in V(D)}{d^-(x)}=\sum_{v \in V(D)}{d^+(x)} \le 3|V(D)|\] and thus we have $d^+(x)=d^-(x)=3$ for all $x \in V(D)$. Consequently, the underlying simple graph $U(D)$ has maximum degree $\Delta(U(D)) \leq 6$. If $U(D)$ is $6$-choosable, then it follows as in the proof of Theorem~\ref{lchrom6} that $D$ is OD-$3$-choosable, a contradiction. 

Therefore, by the list coloring version of Brook's Theorem \cite{listbrooks}, we must have $U(D)=K_7$. Since $D$ is $3$-out- and $3$-in-regular, it follows that $D$ is a tournament on $7$ vertices. 
However, every tournament on $7$ vertices has list dichromatic number at most $3$ and is therefore OD-$3$-choosable according to Theorem~\ref{ldichrom3}. This can be seen using two results from \cite{listdichrom}. Clearly, we have $\vec{\chi}(D) \le 3$. Now if $\vec{\chi}(D)=3$, then we have $|V(D)|=7 \leq 2\vec{\chi}(D)+1$ and by Theorem 2.2 in \cite{listdichrom}, we conclude that $\vec{\chi}_\ell(D)=\vec{\chi}(D)=3$. Otherwise, we have $\vec{\chi}(D) \leq 2$. In this case, we can apply Theorem 3.3 in \cite{listdichrom} to conclude $\vec{\chi}_\ell(D) \leq 2\ln(7) < 4$. Therefore we have $\vec{\chi}_\ell(D) \le 3$ in each case.

Finally, since we obtained that $D$ is OD-$3$-choosable in each case, the initial assumption was wrong, which concludes the proof by contradiction.
\end{proof}
\begin{corollary}\label{aux}
Let $D$ be a digraph with $\min\{d^+(x),d^-(x)+2\} \leq 4$ for every $x \in V(D)$. Then $D$ is majority $3$-choosable.
\end{corollary}
\begin{proof}
For a proof by contradiction, suppose the claim was false and consider a counterexample $D$ minimizing the number of edges. 

Consider first the case that there is a $v \in V(D)$ with $d^+(v)=4$. 
Let $e$ be an edge leaving $v$ and put $D':= D-e$. By the minimality of $D$, $D'$ is majority $3$-choosable. We now claim that any majority-coloring of $D'$ also defines a majority-coloring of $D$. Clearly, such a coloring satisfies the condition for a majority-coloring at any vertex distinct from $v$. Since $v$ has out-degree $3$ in $D'$, it has at most one out-neighbor in $D'$ of the same color. Thus there are at most two out-neighbors of $v$ in $D$ which share its color, and so the majority condition is fulfilled at $v$. We conclude that also $D$ must be majority $3$-choosable, which gives the desired contradiction.

Now for the second case, assume that no vertex has out-degree $4$. This means that for every $x \in V(D)$, we either have $d^+(x) \leq 3$ or $d^+(x) \ge 5$ and therefore $d^-(x) \leq 2$. We can therefore apply Lemma~\ref{outdegree3} to $D$, which shows that $D$ is OD-$3$-choosable. From Theorem~\ref{theorem:odd3listpartition} we get that $D$ is majority $3$-choosable. This again is a contradiction to $D$ being a counterexample to the claim. 

Therefore the initial assumption was wrong, and this concludes the proof.
\end{proof}

\begin{proof}[Proof of Theorem~\ref{maxdegree}.]
If $\Delta^+(D) \leq 4$ or $\Delta(D) \leq 7$, then the claim follows by applying Corollary~\ref{aux}. If $\Delta(U(D)) \leq 6$, then by the list coloring version of Brook's Theorem either $U(D)$ is $6$-choosable, and then the claim follows from Theorem~\ref{lchrom6}, or $U(D)=K_7$.

Now let $L(v_1), \ldots, L(v_7)$ be lists of size three assigned to the vertices $\{v_1,\ldots,v_7\}$ of $D$. 
We first consider the case that all lists are equal, i.e., show that $D$ is majority $3$-colorable.

If there exists a vertex $v \in V(D)$ which is contained in at most $3$ digons, then there are vertices $u_1 \neq u_2 \in V(D) \setminus \{v\}$ such that $u_1,u_2,v$ do not form a directed triangle. Therefore, any partition $\{X_1,X_2,X_3\}$ of $V(D)$ where $X_1=\{v,u_1,u_2\}$ and $|X_2|=|X_3|$ shows, by Theorem~\ref{theorem:3oddpartition}, that $D$ is majority $3$-colorable. 
Otherwise, every vertex in $D$ is contained in at least $4$ digons and thus has out-degree at least $4$. Now any $3$-coloring of $D$ with color classes of sizes $2,2,3$ defines a majority-coloring of $D$. 

Now suppose that not all lists are equal. In this case we can choose for each vertex $v_i$ a sublist $L_2(v_i) \subseteq L(v_i)$ of size two such that no three vertices are assigned the same sublist (minimize the number of edges whose ends are assigned the same sublist). By Theorem~\ref{theorem:2choosability} we obtain a majority-coloring $c$ of $D$ where $c(v_i)$ is contained in $L_2(v_i) \subseteq L(v_i)$.
Hence, $D$ is majority $3$-choosable in each case, which concludes the proof.
\end{proof}

\section{Fractional Majority Colorings}\label{sec:fractional}

Another concept introduced in \cite{kreutzer} is that of a fractional majority coloring. Given a subset $S\subseteq V(D)$, a vertex $v$ is {\it popular in }$S$ if $v\in S$ and more than half of its out-neighbors are in $S$. A subset $S\subseteq V(D)$ is {\it stable} if it contains no popular elements. Let $S(D)$ be the set of all stable sets of $D$, and $S(D,v)$ the set of all stable sets containing $v$. A {\it fractional majority coloring} is a function that assigns a weight $w_T\geq 0$ to every set $T\in S(D)$, satisfying $\sum_{T\in S(D,v)}w_T\geq 1$ for every $v\in V(D)$. The total weight of a fractional majority coloring is simply $\sum_{T\in S(D)}w_T$. Kreutzer et al.~asked for the minimum constant $K$ such that every digraph admits a fractional majority coloring with total weight at most $K$.

We will show two results related to this question, namely Theorem \ref{4-c} and Theorem \ref{fracldeg}. The proof of these two theorems will be based on the dual of the linear program defined by the restrictions on a fractional majority coloring:

\begin{observation}\label{duality} For a digraph $D$, the minimum possible total weight of a fractional majority coloring is also the maximum total weight $\sum_{v\in V(D)} w_v$ in a non-negative weight assignment of $V(D)$ in which every stable set $T$ satisfies $\sum_{v\in T}w_v\leq 1$.\end{observation}

The main idea of the proof of both theorems is that, given any choice of weights on $V(D)$, we can construct a stable set in which the weight is at least a given fraction of the total weight, using the probabilistic method.

\begin{lemma}\label{alterations} Let $D$ be a digraph and let $0<p<1$. Suppose that one can take a random subset $X\subseteq V(D)$ with the property that, for every $v\in V(D)$, the probability that $v$ is in $X$ but not popular in $X$ is at least $p$. Then $D$ admits a fractional majority coloring with total weight at most $\frac1p$.\end{lemma}

\begin{proof}
Suppose that $D$ is a counterexample to our statement, and we will reach a contradiction. By Observation~\ref{duality}, we can assign weights to $V(D)$ so that the total weight is $w>\frac{1}{p}$, and every stable set in $D$ has a sum of weights at most one. Let $Y$ be the set of popular vertices in $X$. By linearity of expectation, the expected total weight of $X\setminus Y$ is at least $pw>1$. 

Take an instance of $X\setminus Y$ with weight greater than 1. Every vertex in $X\setminus Y$ has at least half of its out-neighbors outside of $X$, which implies that it is not popular in $X\setminus Y$. Hence $X\setminus Y$ is stable in $D$ and has total weight greater than 1, producing a contradiction. 
\end{proof}

The proof of Theorem~\ref{fracldeg} is a straightforward application of this lemma:

\begin{proof}[Proof of Theorem~\ref{fracldeg}] Let $N$ be a large enough positive integer. Let $D$ be a digraph with $\delta^+(D)\geq N$. Set $p=\frac12-\sqrt{\frac{\ln N}{N}}$. Let $X$ be a random subset of $V(D)$ in which every element is included independently with probability $p$. By Hoeffding's inequality, for any vertex $v$ the probability that at least half of its out-neighbors are in $X$ is at most \[\Pr\left(|X\cap N^+(v)|\geq\frac12d^+(v)\right)\leq e^{-2\left(\frac12-p\right)^2d^+(v)}\leq e^{-2\ln N}=N^{-2}\]

Setting $q=N^{-2}$, from Lemma~\ref{alterations} we find a fractional majority coloring of total weight at most $\frac{1}{p-q}=2+O\left(\sqrt{\frac{\ln N}{N}}\right)$. For $N \ge C(1/\varepsilon)^2\ln(2/\varepsilon)$, we have $\frac{1}{p-q}<2+\varepsilon$.\end{proof}

For Theorem~\ref{4-c}, we need to be more careful. Consider again the set $X$ containing each vertex independently with probability $p$, where $p$ is slightly lower than $\frac12$. If the out-degree of $v$ is not 1, one can show that the probability that $v$ is popular in $X$ is upper-bounded by a constant, strictly smaller than $p-\frac14$. However, if $v$ has out-degree 1, the probability that $v$ is popular in $X$ is $p^2>p-\frac14$. For this reason, the vertices with out-degree 1 deserve extra consideration.

Observe that, in the graph induced by the vertices of out-degree 1, all cycles are directed, pairwise disjoint and act as sinks. Consequently, removing one vertex from each directed cycle produces an acyclic graph, where the vertices can be given an ordering in which every edge goes from a larger vertex to a smaller one.

\begin{proof}[Proof of Theorem~\ref{4-c}]

Set $p_1=0.4594$ and $p_2=0.4503$. Assign independently to each vertex $v$ a random indicating variable $X_v$, which takes the value $1$ with probability $p_1$ if $d^+(v)=1$ and with probability $p_2$ otherwise. Now construct the random subset $X$ as follows:

\begin{itemize}
\item Add to $X$ all vertices $v$ with $d^+(v)\neq 1$ and $X_v=1$.
\item For every cycle $C$ formed by vertices with $d^+(v)=1$ and $X_v=1$, select a vertex $v\in C$ uniformly at random and set $X_v=0$. 
\item Take an ordering of the vertices $v$ with $d^+(v)=1$ and $X_v=1$, in which if we have an edge $(v,w)$ then $v$ comes after $w$ (this is possible because these vertices form an acyclic digraph). Following this order, add $v$ to $X$ if its out-neighbor is not in $X$.
\end{itemize}

We will show that, for every vertex $v$, the probability that $v$ is in $X$ but not popular in $X$ is at least $\frac14+\varepsilon$, for a fixed value of $\varepsilon>0$. Suppose first that $d^+(v)=1$. If we draw the vertices with out-degree 1 in red and those with other out-degrees in blue, then the successive out-neighborhoods of $v$ must have one of these forms:

\begin{figure}[h]
\begin{center}
\includegraphics[scale=1]{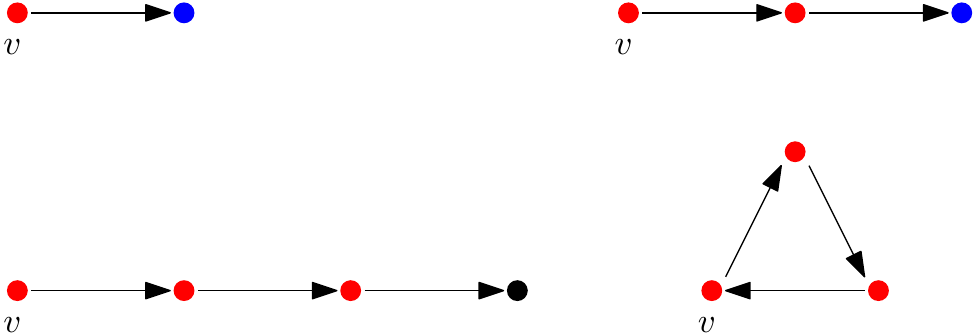}
\end{center}
\caption{The four possible out-neighborhoods of a red vertex. The black vertex here can be either red or blue.} \label{patterns}
\end{figure}

We label the cases as Case 1 through Case 4, left to right and top to bottom in Figure~\ref{patterns}. We denote $v=v_0$, and $v_{i+1}$ as the out-neighbor of $v_i$, if it is unique. We go through each case:

\begin{itemize}
\item If $v$ is in Case 1, then whenever $X_v=1$ and $X_{v_1}=0$ we have $v\in X$. This happens with probability $p_1(1-p_2)$.
\item If $v$ is in Case 2, then whenever $X_v=1$ and $X_{v_1}=0$, or whenever $X_v$, $X_{v_1}$ and $X_{v_2}$ all equal 1, we have $v\in X$. This happens with probability $p_1(1-p_1)+p_1^2p_2$.
\item If $v$ is in Case 3, then whenever $X_v=1$ and $X_{v_1}=0$, or whenever $X_v=1$, $X_{v_1}=1$, $X_{v_2}=1$ and $X_{v_3}=0$ we have $v\in X$. This happens with probability at least $p_1(1-p_1)+p_1^3(1-p_1)$.
\item If $v$ is in Case $4$, if $X_v=1$ and $X_{v_1}=0$, or if $X_v$, $X_{v_1}$ and $X_{v_2}$ all initially equal 1 and $X_{v_1}$ is selected to be modified, then we have $v\in X$. This happens with probability $p_1(1-p_1)+\frac13p_1^3$.
\end{itemize}

Suppose now that $d^+(v)\neq 1$. The probability that $v\in X$ is $p_2$.  If $v$ is popular in $X$, then over half of its out-neighbors $w$ have $X_w=1$ (this is necessary for $w\in X$). Since the $X_w$ are independent, and each of them takes the value 1 with probability at most $p_1$, the probability that $v$ is popular on $X$, conditioned on $v\in X$, is at most $\Pr\left(B(d^+(v),p_1)>\frac{d^+(v)}{2}\right)$. For $d^+(v)=3$, this probability is $3p_1^2-2p_1^3$. We claim that this is the worst case: 

\begin{proposition} For every $k\neq 1$, we have \[\Pr\left(B(k,p_1)>\frac k2\right)\leq\Pr(B(3,p_1)\geq 2).\]\end{proposition}

\begin{proof} Consider an infinite sequence $X_1, X_2, \dots$ of indicating random variables, each taking value 1 independently with probability $p_1$. Let $I_i$ be the event ``among the first $i$ variables more than half take value 1''. Then $\Pr(I_k)=\Pr\left(B(k,p_1)>\frac k2\right)$. Clearly $\Pr(I_0)=0$. Moreover, if $k$ is even then $I_k$ implies $I_{k+1}$, so we can restrict ourselves to odd $k$.

We will prove our statement by induction, by showing that $\Pr(I_{2k+1})<\Pr(I_{2k-1})$ for  $k\geq 2$. Indeed, the event $I_{2k-1}\setminus I_{2k+1}$ is precisely the case in which exactly $k$ of the first $2k-1$ variables take value $1$, and $X_{2k}=X_{2k+1}=0$. Thus $\Pr(I_{2k-1}\setminus I_{2k+1})={2k-1 \choose k}p_1^k(1-p_1)^{k+1}$. Similarly, the event $I_{2k+1}\setminus I_{2k-1}$ is precisely the case in which exactly $k-1$ of the first $2k-1$ variables take value 1, and $X_{2k}=X_{2k+1}=1$. Thus $\Pr(I_{2k+1}\setminus I_{2k-1})={2k-1 \choose k-1} p_1^{k+1}(1-p_1)^k$. Now $P(I_{2k-1})-P(I_{2k+1})=\Pr(I_{2k-1}\setminus I_{2k+1})-\Pr(I_{2k+1}\setminus I_{2k-1})={2k-1 \choose k}p_1^k(1-p_1)^k(1-2p_1)>0$.\end{proof}

With this, we know that for every vertex $v$ the probability that $v$ is in $X$ and not popular in $X$ is at least \begin{align*}&\min\{p_1(1-p_2),p_1(1-p_1)+p_1^2p_2,p_1(1-p_1)+p_1^3(1-p_1),p_1(1-p_1)+\frac13p_1^3,p_2(1-3p_1^2+2p_1^3)\}\\ &=p_2(1-3p_1^2+2p_1^3)=0.252513=:p.  \end{align*}

Applying Lemma~\ref{alterations}, there is a fractional majority coloring of $D$ with total weight at most $\frac 1p<3.9602$.\end{proof}

\section{NP-Hardness of Majority $2$-Coloring} \label{sec:npcompleteness}
The authors of \cite{kreutzer} asked whether there is a polynomial time algorithm to recognize digraphs which have a majority $2$-coloring. In this section, we prove Theorem~\ref{thm:nphardness} and therefore answer this question in the negative by showing that the recognition is NP-complete. For this purpose, we reduce from the problem of testing 2-colorability of $3$-uniform hypergraphs. 

\begin{problem}[2-Coloring 3-Uniform Hypergraphs] \label{prob:hypcoloring}
Input: A $3$-uniform hypergraph $H=(V,E)$. 

Decide whether $V$ can be $2$-colored such that no edge $e \in E$ has all vertices of the same color. 
\end{problem}

It is well-known that the above decision problem is NP-complete, see for instance~\cite{Garey}.

\begin{figure}[h]
\begin{center}
\includegraphics[scale=1]{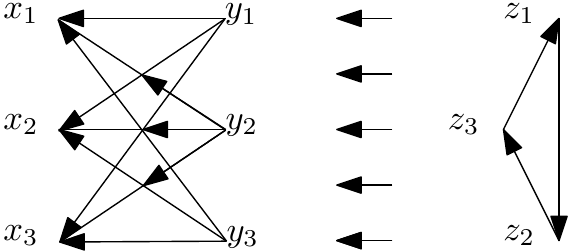}
\end{center}
\caption{The digraph $D_9$. The $18$ edges pointing from the triangle to the remaining six vertices are only indicated. } \label{gadget}
\end{figure}

Let $D_9$ be the digraph with $9$ vertices and $30$ edges obtained from the disjoint union of an upward-orientation of $K_{3,3}$ and a $\vec{C}_3$ by inserting all arcs starting in a vertex of the triangle and ending in one of the six vertices of the oriented $K_{3,3}$. Consider the vertex-labeling of $D_9$ as depicted in Figure~\ref{gadget}. We have the following observation.

\begin{observation}\label{obs:example}
Let $p:\{x_1,x_2,x_3\} \rightarrow \{1,2\}$ be a partial vertex-coloring of $D_9$. Then $p$ can be extended to a majority-coloring of $D_9$ using colors $1$ and $2$ if and only if $p(x_i) \neq p(x_j)$ for some $i, j \in \{1,2,3\}$.
\end{observation}
\begin{proof}
By symmetry and relabeling, it suffices to show that the partial coloring $p(x_1)=p(x_2)=p(x_3)=1$ can not be extended to a majority $2$-coloring of $D_9$, while there is an extension of the coloring $p(x_1)=p(x_2)=1, p(x_3)=2$.

For the first part, suppose there was a majority $2$-coloring $c$ of $D_9$ such that $c(x_1)=c(x_2)=c(x_3)=1$. Then, since $N^+(y_j)=\{x_1,x_2,x_3\}$ for $j=1,2,3$, we must have $c(y_1)=c(y_2)=c(y_3)=2$. Since we are only coloring with two colors, two vertices of the triangle $z_1z_2z_3$ must be colored the same, say $c(z_1)=c(z_2)$. Then when looking at the vertex $z_1$, we see that it has exactly $7$ out-neighbors, of which at least $4$ also have color $c(z_1)$. This shows that $c$ is no valid majority-coloring, a contradiction.

For the second part, it suffices to verify that the coloring $c(x_1)=c(x_2)=1, c(x_3)=2$, $c(y_1)=c(y_2)=c(y_3)=2$, $c(z_1)=c(z_2)=c(z_3)=1$ defines a majority $2$-coloring of $D_9$.
\end{proof}

\begin{proof}[Proof of Theorem~\ref{thm:nphardness}.]
The NP-membership follows because a valid majority $2$-coloring can be verified in polynomial time in the size of $D$.

For the NP-hardness we describe a polynomial reduction from Problem~\ref{prob:hypcoloring}. Suppose we are given an arbitrary $3$-uniform hypergraph $H=(V,E)$ as an instance for Problem~\ref{prob:hypcoloring}. For each edge $e \in E$ produce a distinct copy of $D_9$. Now identify the vertices of every edge $e$ with the three vertices $x_1,x_2,x_3$ of its respective $D_9$-copy. This gives rise to a digraph $D$ consisting of $|V|+6|E|$ vertices in which several copies of $D_9$ share certain sink vertices (namely, if their corresponding hyperedges in $H$ intersect). 

Because we only identify sink vertices of the $D_9$-copies, a $2$-coloring of the vertices of $D$ is a majority $2$-coloring if and only if the induced $2$-coloring on each of the $D_9$-copies is a majority $2$-coloring. By Observation~\ref{obs:example} we conclude that $D$ admits a majority $2$-coloring if and only if it is possible to color the vertices of $H$ in such a way that in every hyperedge there are two vertices with distinct colors. 

We have therefore reduced Problem~\ref{prob:hypcoloring} with instance $H$ to the problem of deciding whether the polynomial-sized digraph $D$ is majority $2$-colorable, which shows the correctness of the reduction. This concludes the proof of the claimed NP-hardness of the majority-$2$-coloring problem.
\end{proof}

\section{Conclusive Remarks and Discussion}\label{concl}
Gir\~{a}o et al.~\cite{girao} and independently Knox and \v{S}\'{a}mal \cite{linear} investigated a natural generalization of majority colorings: For any $\alpha \in [0,1]$, define an \emph{$\alpha$-majority coloring} of a digraph $D$ to be a vertex-coloring in which for every vertex $v$, at most $\alpha \cdot d^+(v)$ vertices in $N^+(v)$ have the same color as $v$. If such a coloring can be found for any assignment of lists of size at least $\ell$ to the vertices, call the digraph \emph{$\alpha$-majority $\ell$-choosable}.

Generalizing the result by Anholcer et al.~it was proved both in \cite{girao} and \cite{linear} that for every integer $k \ge 1$, every digraph is $\frac{1}{k}$-majority $2k$-choosable. Gir\~{a}o et al.~proposed the following generalization of Conjecture~\ref{mainconj}:
\begin{conjecture}\label{2kconjecture}
For every integer $k \ge 1$, every digraph $D$ has a $\frac{1}{k}$-majority $(2k-1)$-coloring. In fact, every digraph is $\frac{1}{k}$-majority $(2k-1)$-choosable.
\end{conjecture}

It is natural to try and generalize the results presented in this paper for majority colorings with $\alpha=\frac{1}{2}$ to arbitrary values $\alpha \in [0,1]$. Among our results, we can only generalize a special case of Theorem~\ref{dichrom3}, namely for digraphs of dichromatic number $2$, we verify the first part of Conjecture~\ref{2kconjecture} for all $k \ge 1$.
\begin{proposition}\label{proposition:generalizedcoloring}
Let $D$ be a digraph with $\vec{\chi}(D) \leq 2$. Then for every $k \in \mathbb{N}$, $k \ge 2$, $D$ admits a $\frac{1}{k}$-majority coloring using $2k-1$ colors.
\end{proposition}
\begin{proof}
Consider first an acyclic digraph $F$ with a pre-coloring of its sinks using colors from $\{1,\ldots,k\}$. We claim that such a coloring can always be extended to a $\frac{1}{k}$-majority coloring of $F$ also using colors from $\{1,\ldots,k\}$. To find such a coloring, we take a topological ordering $x_1,\ldots,x_n$ of the vertices (i.e. $(x_i,x_j) \notin E(D)$ for all $i \le j$) such that $\{x_1,\ldots,x_t\}$ are the pre-colored sinks. Now we color the vertices one by one, starting with $x_{t+1}$, then $x_{t+2}$ etc. When coloring the vertex $x_i$ with $i>t$, we assign to it a color from $\{1,\ldots,k\}$ appearing least frequently among its (already colored) out-neighbors. This procedure eventually yields a $k$-coloring of $F$ where any vertex has at most a $\frac{1}{k}$-fraction of its out-neighbors with the same color. 

Now let $\{X_1,X_2\}$ be a partition of $V(D)$ such that $D[X_1]$, $D[X_2]$ are acyclic. For $i=1,2$ let $D_i'$ be the digraph obtained from $D[X_i]$ by adding all arcs in $D$ leaving $X_i$ together with their endpoints. Clearly, also $D_1'$ and $D_2'$ are acyclic. By the above observation, $D_i'$ for $i=1,2$ has a majority $\frac{1}{k}$-coloring $c_i$ with $k$ colors in which all sinks receive color $1$. After renaming we may suppose that $c_1$ uses colors from $\{1,2,\ldots,k\}$, while $c_2$ uses colors from $\{1,k+1,k+2,\ldots,2k-1\}$. We now define a $(2k-1)$-coloring of all vertices in $D$ by putting $c(x):=c_i(x)$ for $x \in X_i$. For any vertex $x \in X_i$, we have that $N^+(x)=N_{D_i'}^+(x)$, and, since all vertices in $V(D_i') \setminus X_i$ received color $1$ under $c_i$, it follows that $\{y \in N^+(x)|c(y)=c(x)\} \subseteq \{y \in N_{D_i'}^+(x)|c_i(y)=c_i(x)\}$. Therefore, and since $c_i$ is a majority $\frac{1}{k}$-coloring of $D_i'$, at most a $\frac{1}{k}$-fraction of vertices in $N^+(x)$ have the same color as $x$. This shows that $c$ is a coloring as requested and concludes the proof. 
\end{proof}
It is worth noting that the above bound is tight. Consider for example the circulant digraph $\vec{C}(2k-1,k)$ which has as vertex set $\mathbb{Z}_{2k-1}$, and where we have an edge $(i,j)$ if and only if $j-i \in \{1,2,3\ldots,k-1\}$. It is easy to see that in any majority $\frac{1}{k}$-coloring of $D$, the $2k-1$ vertices must receive pairwise distinct colors, however, the partition $X_1=\{0,1,\ldots,k-1\}$, $X_2=\{k,k+1,\ldots,2k-2\}$ of the vertex set shows that $\vec{\chi}(\vec{C}(2k-1,k))=2$.

The methods used in this paper are unlikely to resolve Conjecture~\ref{mainconj} for the open cases of $5$- and $6$-regular digraphs. One possible approach could be via an extension to hypergraphs: Given a $5$-regular digraph $D$, consider the hypergraph $\mathcal{H}(D)$ with vertex set $V(D)$ and whose edges are $\{v\} \cup N^+(v)$, $v \in V(D)$. This hypergraph is $6$-regular and $6$-uniform. If we could now find a vertex-$3$-coloring of $\mathcal{H}(D)$ such that no hyperedge contains four vertices of the same color, this coloring would certainly be a majority coloring of $D$. We would therefore be interested in deciding the following question.
\begin{problem}
Let $H$ be a $6$-regular $6$-uniform hypergraph. Is there a $3$-coloring of $V(H)$ such that no hyperedge contains four vertices of the same color?
\end{problem}
The setting of $k$-regular $k$-uniform hypergraphs could be fruitful, as it is known that these hypergraphs have property B for all $k \ge 4$ (as noted in \cite{viswanathan}). We want to conclude with a small selection of open questions.
\begin{itemize}
\item Is every $5$-regular digraph $\frac{1}{3}$-majority $5$-colorable? We can show that it is possible to color with $5$ colors such that in each connected component, at most one vertex violates the majority condition.
\item Does every digraph with $\chi(D) \leq 6$ have a $\frac{1}{3}$-majority $5$-coloring?
\item Does every digraph $D$ with $\vec{\chi}(D) \leq 3$ have a $\frac{1}{k}$-majority $(2k-1)$-coloring for every $k \ge 1$?
\end{itemize}
\bibliography{references}
\bibliographystyle{alpha}

\end{document}